%% file: main.tex
\documentclass[11pt, reqno]{amsart}

\input{mypreamble}

\usepackage{booktabs, multirow, dcolumn}
\usepackage{tikz}

\title{$K_4$-intersecting families of graphs}
\author[Berger]{Aaron Berger}
\author[Zhao]{Yufei Zhao}
\address{Department of Mathematics, Massachusetts Institute of Technology, Cambridge, MA 02139, USA}
\email{\{bergera,yufeiz\}@mit.edu}
\date{}
\thanks{Berger was supported by NSF Graduate Research Fellowship DGE-1745302.
Zhao was supported by NSF Award DMS-1764176, the MIT Solomon Buchsbaum Fund, and a Sloan Research Fellowship.}

\renewcommand{\wh}{\widehat}
\renewcommand{\subset}{\subseteq}

\newcommand{\N}{\bb N}

\DeclareMathOperator{\Unif}{Unif}

\begin{document}

\begin{abstract}
    Ellis, Filmus, and Friedgut proved an old conjecture of Simonovits and S\'os showing that the maximum size of a triangle-intersecting family of graphs on $n$ vertices has size at most $2^{\binom{n}{2} - 3}$, with equality for the family of graphs containing some fixed triangle. They conjectured that their results extend to cross-intersecting families, as well to $K_t$-intersecting families. We prove these conjectures for $t \in \{3,4\}$, showing that if $\mc F_1$ and $\mc F_2$ are families of graphs on $n$ labeled vertices such that for any $G_1 \in \mc F_1$ and $G_2 \in \mc F_2$, $G_1 \cap G_2$ contains a $K_t$, then $\lvert \mathcal F_1 \rvert  \lvert \mathcal F_2 \rvert \le 4^{\binom{n}{2} - \binom{t}{2}}$, with equality if and only if $\mc F_1 = \mc F_2$ consists of all graphs that contain some fixed $K_t$. 
    We also establish a stability result. 
    More generally, ``$G_1 \cap G_2$ contains a $K_t$'' can be replaced by ``$G_1$ and $G_2$ agree on a non-$(t-1)$-colorable graph.''
\end{abstract}

\maketitle

\section{Introduction}

We say that a family $\mc F$ of graphs on a common vertex set is \emph{triangle-intersecting} if, for every $G_1, G_2 \in \mc F$, $G_1 \cap G_2$ contains a triangle. 
A classic problem due to Simonovits and S\'os from 1976 asks to determine the largest triangle-intersecting family of graphs on $n$ labeled vertices.
They conjectured that the maximizing family is obtained by taking all graphs containing some fixed triangle, which has size $2^{\binom{n}{2} - 3}$.
Triangle-intersecting families are a graph theoretic analogue of intersecting set families, which have been extensively studied, with the Erd\H{o}s--Ko--Rado theorem \cite{EKR61} being a fundamental result.
On the other hand, intersecting graph families are much less understood, and many standard techniques for studying intersecting set families (e.g., shifting) do not easily adapt to the graph setting.

The first significant progress on the Simonovits--S\'os conjecture was due to Chung, Graham, Frankl, and Shearer \cite{CGFS86}, who introduced a powerful entropy lemma (now commonly known as Shearer's lemma) and used it to prove an upper bound of $2^{\binom{n}{2} - 2}$, though it misses the conjecture by a factor of $2$. A significant breakthrough was obtained by Ellis, Filmus, and Friedgut~\cite{EFF12}, who used Fourier analytic methods to prove the Simonovits--S\'os conjecture along with several strengthenings. 

We say that a family $\mc F$ of graphs on a common vertex set is \emph{$H$-intersecting} if for every $G_1, G_2 \in \mc F$, $G_1 \cap G_2$ contains $H$ as a subgraph. Here is a natural extension of the above problem.

\begin{problem}\label{prob:intersecting}
Given $H$ and $n$, determine the maximum size of an $H$-intersecting family of graphs on $n$ labeled vertices.
\end{problem}

A family of graphs on a common set of vertices is called an \emph{$H$-umvirate} if the family consists of all graphs containing some fixed copy of $H$. 
Clearly an $H$-umvirate family on $n$ vertices is $H$-intersecting and has size $2^{\binom{n}{2} -e(H)}$. 
For a fixed $H$, is an $H$-umvirate family is the largest $H$-intersecting family?
It turns out that the answer can be no (see discussions at the end of this section), but it is conjectured \cite{EFF12} that the answer is yes whenever $H$ is a clique.

\begin{conjecture} \label{conj:clique-intersecting}
Every $K_t$-intersecting family of graphs on $n$ vertices has size at most $2^{\binom{n}{2} - \binom{t}{2}}$.
\end{conjecture}

Ellis, Filmus, and Friedgut proved the result for $t = 3$. 
Applying tools developed by Friedgut \cite{Fri08},
they also proved uniqueness and stability of the maximizer, showing that a triangle-intersecting family with size nearly the maximum size must be close to a triangle-umvirate.

\begin{theorem}[{\cite[Theorem 1.4]{EFF12}}]\label{thm:EFF} Let $\mc F$ be a triangle-intersecting family of graphs on $n$ labeled vertices. 
	\begin{itemize}
	    \item (\textbf{Upper bound}) $\abs{\mc F} \le 2^{\binom{n}{2}-3}$
	    \item (\textbf{Uniqueness}) The upper bound is an equality if and only if $\mc F$ is a triangle-umvirate.
	    \item (\textbf{Stability}) 
	    There exists an absolute constant $C> 0$ so that for all $\epsilon > 0$, 
	    if $\abs{\mc F} \ge (1-\epsilon) 2^{\binom{n}{2}-3}$, then there exists a triangle-umvirate $\mc T$ such that $\abs{\mc T \triangle \mc F} \le C \epsilon 2^{\binom{n}{2}}$. (Here $\triangle$ denotes the set symmetric difference.)
	\end{itemize}
\end{theorem}

Results about intersecting families of sets are often extended to cross-intersecting families, and \cite{EFF12} conjectures that the same extensions hold here as well.

\begin{definition}[Cross intersecting families]
A pair $(\mc F_1, \mc F_2)$ of families of graphs on a common vertex set is said to be \emph{cross-$H$-intersecting} if for any $G_1 \in \mc F_1, G_2 \in \mc F_2$ one has that $G_1 \cap G_2$ contains a subgraph isomorphic to $H$. 
\end{definition}

Conjecture 1 in \cite{EFF12} says that if $(\mc F_1, \mc F_2)$ is a cross-triangle-intersecting pair of family of graphs on $n$ labeled vertices, then $\abs{\mc F_1} \abs{\mc F_2} \le 4^{\binom{n}{2} - 3}$, with equality if and only if $\mc F_1 = \mc F_2$ is a triangle-umvirate.
We prove this conjecture along with its $K_4$-intersecting analogue.
Our solution applies the framework set up in \cite{EFF12}, and where they need to verify certain inequalities by hand via casework, we come up with systematic way to potentially verify the conjecture for any fixed $t$ via a finite computation, though the complexity of the computation increases extremely quickly with $t$.

\begin{theorem}\label{thm:main-intersecting} 
Let $t \in \{3,4\}$. Let $(\mc F_1, \mc F_2)$ be a cross-$K_t$-intersecting pair of families graphs on $n$ labeled vertices.
	\begin{itemize}
	    \item (\textbf{Upper bound})
	    $\abs{\mc F_1}\abs{\mc F_2} \le 4^{\binom{n}{2} - \binom{t}{2}}$
	    \item (\textbf{Uniqueness}) The upper bound is an equality if and only if $\mc F_1 = \mc F_2$ is a $K_t$-umvirate.
	    \item (\textbf{Stability}) There exists a constant $C_t> 0$ depending only on $t$ so that for all $\epsilon > 0$, if $\abs{\mc F_1}\abs{\mc F_2} \ge (1 - \epsilon) 4^{\binom{n}{2} - \binom{t}{2}}$,
	    then there exists a $K_t$-umvirate $\mc T$ such that $\abs{\mc F_1 \triangle \mc T}, \abs{\mc F_2 \triangle \mc T} \le C_t \epsilon 2^{\binom{n}{2}}$.
	\end{itemize}
\end{theorem}

\begin{conjecture}
Theorem \ref{thm:main-intersecting} holds for all $t \ge 5$.
\end{conjecture}

\cref{thm:EFF} was proved in \cite{EFF12} in a more general context of odd-cycle-agreeing families. 
Our results also hold in this generality.
We write $G_1 \triangle G_2$ for the edge symmetric difference of two graphs, and $\ol G$ for the edge-complement of a graph.

\begin{definition}
A pair $(\mc F_1, \mc F_2)$ of families of graphs on a common vertex is said to be \textit{cross-$t$-chromatic-intersecting} if for any $G_1 \in \mc F_1$, $G_2 \in \mc F_2$, $G_1 \cap G_2$ is not $(t-1)$-colorable, and is said to be \emph{cross-$t$-chromatic-agreeing} if for any $G_1 \in \mc F_1$, $G_2 \in \mc F_2$, $\ol{G_1 \triangle G_2}$ is not $(t-1)$-colorable (equivalently, there is some $t$-chromatic graph on which $G_1$ and $G_2$ agree).
\end{definition}

A cross-$K_t$-intersecting pair is clearly cross-$t$-chromatic-agreeing. As stated below, \cref{thm:main-intersecting} holds also for this stronger notion, and we conjecture that it holds for all $t$.
Here the equality case needs to be changed from being a $K_t$-umvirate to an \emph{$K_t$-conjunction}, which is defined to be a family of all graphs with a prescribed intersection with some given copy of $K_t$. 
In other words, $\mc F$ is a $K_t$-conjunction if and only if there is some graph $H$ on the same common vertex set such that $\{G \triangle H : G \in \mc F\}$ is a $K_t$-umvirate.\footnote{\cite{EFF12} uses ``$K_t$-junta'' to mean what we are calling $K_t$-conjunction here, which differs from the common usage of ``junta'' to mean any function determined by its restriction on a fixed set.}

\begin{theorem} \label{thm:main-agreeing}
Let $t \in \{3,4\}$. Let $(\mc F_1, \mc F_2)$ be a cross-$t$-chromatic-agreeing pair of families graphs on $n$ labeled vertices.
	\begin{itemize}
	    \item (\textbf{Upper bound})
	    $\abs{\mc F_1}\abs{\mc F_2} \le 4^{\binom{n}{2} - \binom{t}{2}}$
	    \item (\textbf{Uniqueness}) The upper bound is an equality if and only if $\mc F_1 = \mc F_2$ is a $K_t$-conjunction.
	    \item (\textbf{Stability}) There exists a constant $C_t> 0$ depending only on $t$ so that for all $\epsilon > 0$, if $\abs{\mc F_1}\abs{\mc F_2} \ge (1 - \epsilon) 4^{\binom{n}{2} - \binom{t}{2}}$,
	    then there exists a $K_t$-conjunction $\mc T$ such that $\abs{\mc F_1 \triangle \mc T}, \abs{\mc F_2 \triangle \mc T} \le C_t \epsilon 2^{\binom{n}{2}}$.
	\end{itemize}
\end{theorem}
The majority of the paper (Sections \ref{sec:fourier} and \ref{sec:constructions}) is devoted to proving \cref{thm:main-agreeing}. In \cref{sec:additional-proofs} we show how one can deduce \cref{thm:main-intersecting} from \cref{thm:main-agreeing}, along with a $p$-biased version which we will introduce briefly now.

For a family $\mc F$ of graphs on $n$ labeled vertices, let $\mu_p(\mc F) = \sum_{G \in \mc F} p^{e(G)}(1-p)^{e(\overline G)}$.
Ellis, Filmus, and Friedgut \cite{EFF12} showed that \cref{thm:EFF} also holds for $p \le 1/2$ when $|\mc F|$ is replaced by $\mu_p(\mc F)$ with an upper bound of $\mu_p(\mc F) \le p^3$, and the conclusion of (stability) modified accordingly.
It was conjectured that these results extend to $1/2 < p \le 3/4$, but this was disproved by Keller and Lifschitz~\cite{KL19}. Whereas the $p$-biased proof was somewhat involved in \cite{EFF12}, recent work of Ellis, Keller, and Lifshitz \cite{EKL19} enables one to take any of a wide variety of $p = 1/2$ intersecting family results as a black box and deduce $p < 1/2$ versions, including stability. This enables us to obtain the following corollary of \cref{thm:main-agreeing} (see~\cref{sec:additional-proofs}).
\begin{corollary}\label{cor:smaller_p} 
Let $0 < p \le 1/2$ and let $t \in \{3,4\}$. Let $(\mc F_1, \mc F_2)$ be a cross-$t$-chromatic-intersecting pair of families graphs on $n$ labeled vertices.
	\begin{itemize}
	    \item (\textbf{Upper bound})
	    $\mu_p(\mc F_1)\mu_p(\mc F_2) \le p^{-2\binom{t}{2}}$
	    \item (\textbf{Uniqueness}) The upper bound is an equality if and only if $\mc F_1 = \mc F_2$ is a $K_t$-umvirate.
	    \item (\textbf{Stability}) There exists a constant $C_{p,t}> 0$ depending only on $p$ and $t$ so that for all $\epsilon > 0$, if $\abs{\mc F_1}\abs{\mc F_2} \ge (1 - \epsilon) p^{-2\binom{t}{2}}$,
	    then there exists a $K_t$-umvirate $\mc T$ such that $\mu_p(\mc F_i \sm \mc T) \le C_{p,t}\epsilon^{\log_{(1-p)}(p)}$ for $i \in \{1,2\}$.
	\end{itemize}
\end{corollary}

Let us mention some additional related results.
As noted at the end of of \cite{EFF12}, the natural generalization of \cref{conj:clique-intersecting} is false a general $H$, as an $H$-intersecting family could be a constant factor larger than an $H$-umvirate.
Indeed, Noga Alon observed that for every fixed star forest $H$, the largest $H$-intersecting family on $n$ vertices has size $(1- o(1)) 2^{\binom{n}{2} -1}$ (divide the $n$ vertices into two equal halves $A \cup B$ and take all graphs with at least $n/4+C$ vertices in $A$ having degree $\ge n/4 + C$ to $B$; this family is $H$-intersecting for a large enough $C$, and it has size $(1- o(1)) 2^{\binom{n}{2} -1}$. On the other hand, by pairing graphs with their complements, the size of an $H$-intersecting family is at most $2^{\binom{n}{2} - 1}$).
He further conjectured that there is some universal constant $c> 0$ such that for $H$ that is not a star forest, the largest $H$-intersecting family on $n$ vertices has size at most $(1-c)2^{\binom{n}{2} - 1}$. 
It would suffice to prove this conjecture for the 3-edge-path $P_3$. 
It had been conjectured \cite{CGFS86} that the largest $P_3$-intersecting family is a $P_3$-umvirate, but this is false, as Christofides constructed a $P_3$-intersecting family of graphs on $6$ vertices with size $17 \cdot 2^8 > 2^{\binom{6}{2} - 3}$.
See \cite{BBDDDMV19, KMS12} for additional related work.

\subsection*{Organization}
In \cref{sec:reduction} we explain the framework introduced in \cite{EFF12} to reduce the main result \cref{thm:main-agreeing} to a certain linear program. The reduction here is essentially the same as in \cite{EFF12}, though we need to state everything for cross-intersecting families instead of a single intersecting family as was done in \cite{EFF12}. 
In \cref{sec:fourier}, we reduce the verification of dual constraints to a finite computation---this is where our arguments start to differ from \cite{EFF12}.
In \cref{sec:constructions}, we construct feasible dual solutions and verify their validity.
In \cref{sec:additional-proofs}, we explain how to deduce \cref{thm:main-intersecting} and \cref{cor:smaller_p} from \cref{thm:main-agreeing}.
In \cref{sec:discussion}, we conclude with some brief remarks on the different possible directions for generalization.

\subsection*{Acknowledgements}
The first author thanks Pat Devlin for introducing him to this problem, and Jonathan Tidor for some helpful discussions.
We thank Noam Lifshitz for pointing out to us that one can deduce \cref{cor:smaller_p} using \cite{EFF12}.

\section{Reduction to a linear program} \label{sec:reduction}

In this section, we explain the solution framework introduced in \cite{EFF12} that reduces the problem to a certain linear program. 
This reduction is analogous to Delsarte's linear programming bound for error correcting codes \cite{Del73}.

\subsection{Fourier analytic bound}

Every graph on $n$ labeled vertices is identified with its edge set indicator vector, viewed as an element of $\F_2^N$ with $N = \binom{n}{2}$. 
Families of graphs correspond to indicator functions $f : \F_2^{N} \to \{0,1\}$. 
We use the following standard conventions for Fourier analysis on $\F_2^N$, normalized with the uniform measure on the physical space and the counting measure on frequency space.
For a function $f: \F_2^N \to \R$, its Fourier transform is given by 
\begin{equation*}
    \wh f(\lambda) = \E_{x \in \F_2^N} f(x)\cdot (-1)^{\lambda \cdot x}, \qquad \lambda \in \F_2^N
\end{equation*}
whereas the Fourier inversion formula is given by
\begin{equation*}
    f(x) = \sum_{\lambda \in \F_2^N} \wh f(\lambda) \cdot (-1)^{\lambda \cdot x}, \qquad x \in \F_2^N.
\end{equation*}
Given two function $f, g \colon \F_2^N \to \R$, we denote their convolution by $f * g \colon \F_2^N \to \R$ by $(f*g)(x) = \E_y f(y) g(x-y)$. We have the identity $\wh {f*g}(\lambda) =\wh f (\lambda) \wh g (\lambda)$.

The following claim can be viewed as a weighted version of Hoffman's bound on the independence number of a regular graph in terms of its eigenvalues.
It is also known as a linear programming bound.

\begin{proposition} \label{prop:Hoffman}
Let $f, g : \F_2^{N} \to \{0,1\}$ and let $\nu: \F_2^{N} \to \R$ satisfy $\E[\nu] = 1$ and $\angles{f * g, \nu} = 0$. Then 
\begin{itemize}
    \item (\textbf{Upper bound})We have
    \begin{equation}\label{eqn:Hoffman}
\E[f]\E[g] \leq \left(\frac{m}{1+m}\right)^2, 
\end{equation}
where $m = \max_{\lambda\neq 0}\abs{\wh \nu(\lambda)}$ is the largest magnitude among nontrivial eigenvalues of $\nu$. 
\item (\textbf{Maximal families}) If equality holds, then $\E[f] = \E[g]$, and $\wh f(\lambda)=\wh g(\lambda) = 0$ for all $\lambda \neq 0$ with $|\wh \nu(\lambda)| < m$.
\item (\textbf{Stability}) For all $\epsilon, \delta \in (0,1]$, if $\E[f]\E[g] \ge (m/(1+m))^2 - \epsilon$, then
\begin{equation*}
    \sum_{\substack{\lambda \neq 0\\|\wh \nu(\lambda)| \le (1-\delta)m}} \wh f(\lambda)^2 \le \frac{(1+m)^4}{\delta(2-\delta) m^2}\epsilon.
\end{equation*}
\end{itemize}

\end{proposition}

\begin{proof}[Proof (Upper bound).]
	Taking a Fourier transform, $0 = \angles{f * g, \nu} = \sum_{\lambda \in \F_2^N} \wh f(\lambda) \wh g(\lambda) \wh \nu(\lambda)$. Since this sum equals 0, the contribution from $\lambda = 0$ must cancel with the rest of the summation:
	\begin{align*}
	\E [f]\E [g]
	&= -\sum_{\lambda \neq 0}  \widehat{\nu}(\lambda)\wh f(\lambda)\wh g(\lambda)\\
	&\leq \max_{\lambda \neq 0}\Big[|\wh \nu(\lambda)|\Big] \sum_{\lambda \neq 0}|\wh f(\lambda)||\wh g(\lambda)|\tag{$*$}\\
	& \le m \Big(\sum_{\lambda \neq 0} \wh f(\lambda)^2\Big)^{\frac 12}
	\Big(\sum_{\lambda \neq 0} \wh g(\lambda)^2\Big)^{\frac 12} \tag*{[Cauchy--Schwarz]}\\
	&= m(\E [f]-\E [f]^2)^{\frac 12}(\E [g]-\E [g]^2)^{\frac 12}\tag*{[Plancherel]}\\
	&= m\left(\E [f]\E [g]-\E [f]^2\E [g]-\E [f]\E [g]^2+\E [f]^2\E [g]^2\right)^{\frac 12}\\
	&\le m\left(\E [f]\E [g]-2\E [f]^{3/2}\E [g]^{3/2}+\E [f]^2\E [g]^2\right)^{\frac 12}\tag*{[AM-GM]}\\
	&= m((\E [f]\E [g])^{1/2} - \E [f]\E [g]).
	\end{align*}
	This is now a quadratic in $(\E [f]\E [g])^{1/2}$; solving for this quantity in terms of $m$ gives the desired result.
\end{proof}
\begin{proof}[Proof (Maximal families).]
For equality to hold, each inequality above must be tight. For the AM-GM step to be an equality, we must have $\E[f] = \E[g]$, assuming neither is identically zero. For the Cauchy--Schwarz step to be an equality, we must have $|\wh f(\lambda)|  = \alpha|\wh g(\lambda)|$ for some $\alpha$ and all $\lambda \neq 0$. Since $\hat f(0) = \E[f] = \E[g] = \hat g(0)$, we deduce $\alpha = 1$. In particular, $\hat f$ and $\hat g$ have the same support, and so if ($*$) is also an equality, we must have $|\wh \nu(\lambda)| = m$ whenever $|\wh f(\lambda)|, |\wh g(\lambda)| > 0$, which completes the proof.
\end{proof}
\begin{proof}[Proof (Stability).] Let $D := \sum_{\lambda \neq 0, |\wh \nu(\lambda)| < m} \wh f(\lambda)^2$. In order to get a tighter bound on $\E[f]\E[g]$, we replace $(*)$ with a tighter bound. Define 
\begin{equation*}
    \widetilde f(\lambda) := \begin{cases}
    \wh f(\lambda) & \lambda = 0 \text{ or } |\wh \nu (\lambda)| > (1-\delta)m\\
    \wh f(\lambda)\cdot(1-\delta) & \text{otherwise}.
    \end{cases}
\end{equation*}
With this definition, we have
\begin{align*}
	\E[f]\E[g]
	&= -\sum_{\lambda \neq 0}  \widehat{\nu}(\lambda)\wh f(\lambda)\wh g(\lambda)\\
	&\le \sum_{\substack{\lambda \neq 0\\|\wh \nu(\lambda)| > (1-\delta)m}} m|\wh f(\lambda)\wh g(\lambda)| + \sum_{\substack{\lambda \neq 0\\|\wh \nu(\lambda)| \le (1-\delta)m}} (1-\delta)m|\wh f(\lambda)\wh g(\lambda)|\\
	&= m \sum_{\lambda \neq 0}|\widetilde f(\lambda)||\wh g(\lambda)|\\
	& \le m \Big(\sum_{\lambda \neq 0} \widetilde f(\lambda)^2\Big)^{\frac 12}
	\Big(\sum_{\lambda \neq 0} \wh g(\lambda)^2\Big)^{\frac 12} \tag*{[Cauchy--Schwarz]}
	\end{align*}
Let us pause for a moment to consider this $\sum_{\lambda \neq 0} \widetilde f(\lambda)^2$ term. We can rewrite this as 
\begin{equation*}
\sum_{\lambda \neq 0} \widetilde f(\lambda)^2 = \sum_{\substack{\lambda \neq 0\\ |\wh \nu(\lambda)| > (1-\delta)m}} \wh f(\lambda)^2 + (1-\delta)^2\sum_{\substack{\lambda \neq 0\\|\wh \nu(\lambda)| \le (1-\delta)m}}\wh f(\lambda)^2 = \sum_{\lambda \neq 0} \wh f(\lambda)^2 - D\delta(2-\delta),
\end{equation*}
where we have plugged in the definition of $D$ from the beginning of the proof.
Since $\sum \wh f(\lambda)^2 \le 1$, the RHS is at most $(1-D\delta(2-\delta))\sum_{\lambda \neq 0} \wh f(\lambda)^2$. This multiplicative factor can now be pulled out and combined with the $m$, and the rest of the proof proceeds as before. To simplify notation, let $c^2 = 1-D\delta(2-\delta)$. Expanding out the final result, we obtain the upper bound
\begin{align*}
    \E[f]\E[g] \le \left(\frac{cm}{1+cm}\right)^2 &=
    \frac{m^2}{(1+m)^2} - \frac{m^2(1+cm)^2 - (1+m)^2(cm)^2}{(1+m)^2(1+cm)^2}\\
    &= \frac{m^2}{(1+m)^2} - \frac{m^2(1-c^2)+ 2m^3(c-c^2)}{(1+m)^2(1+cm)^2}\\
    &\le \frac{m^2}{(1+m)^2} - \frac{m^2(1-c^2)}{(1+m)^4}\\
    &=\left(\frac{m}{1+m}\right)^2 - \frac{D\delta(2-\delta) m^2}{(1+m)^4}.
\end{align*}
Thus $D\delta(2-\delta)m^2/(1+m)^4 \le \epsilon$. Solving for $D$ completes the proof.
\end{proof}

\subsection{A class of dual functions}
In order to apply \cref{prop:Hoffman} to the indicator functions $f,g$ of a cross-$t$-chromatic-agreeing pair of families, we need a suitable $\nu$.
\begin{lemma}\label{lem:nu_support_condition}
Let $f,g$ be indicator functions of a cross-$t$-chromatic-agreeing pair of families of graphs on $n$ labeled vertices. Let $N = \binom{n}{2}$ and $\nu : \F_2^N \to \R$ be supported on graphs whose complements are $(t-1)$-colorable. Then $\angles{f * g, \nu} = 0$.
\end{lemma}
\begin{proof}
Since $f,g$ are cross-$t$-chromatic-agreeing, for any $x$ with $f(x) = 1$ and $y$ with $g(y) = 1$, $x+y$ is the edge-indicator vector of a graph whose complement is not $(t-1)$-colorable. Hence, $\nu(x+y) = 0$. It follows that $\angles{f * g, \nu} = \E_{x,y}[f(x)g(y)\nu(x+y)] = 0$.
\end{proof}

In view of \cref{prop:Hoffman}, it therefore suffices to construct some function $\nu: \F_2^N \to \R$ that is supported on graphs whose complements are $(t-1)$-colorable, with $\E[\nu] = 1$ and $ m = \max_{\lambda\neq 0}\big[|\wh \nu(\lambda)|\big] \le 1/(2^{\binom{t}{2}}-1)$. For stability purposes we would also like to ensure that $|\wh \nu(\lambda)| < m - \delta$ for some fixed $\delta > 0$ whenever $\lambda$ has more than $\binom{t}{2}$ edges.

This condition on the support of $\nu$ is complicated to state, but easy to work with, as it is linear. The following useful class of functions (and therefore, their linear combinations) satisfy this condition.

\begin{lemma}\label{lem:nu_construction_condition} Let $N = \binom{n}{2}$, let $\mc T$ be an arbitrary distribution on $(t-1)$-colorable graphs, and let $\{f_G\}$ be a set of functions indexed by graphs on $n$ vertices. If $\nu:\F_2^N \to \R$ satisfies
$$\wh \nu(\lambda) = (-1)^{e(\lambda)} \E_{T \sim \mc T} f_T(T \cap \lambda),$$ 
then $\nu$ is supported on graphs whose complements are $(t-1)$-colorable.
\end{lemma}

\noindent\textit{Remark.}
This is a slight generalization of {\cite[Lemma 2.8]{EFF12}}, but the proof is identical.\footnote{To obtain this version, replace the words ``OCC Spectrum'' with ``the Fourier transform of some $\nu$ supported on graphs whose complements are $(t-1)$-colorable,'' and instead of a distribution on bipartite graphs take a distribution on $(t-1)$-colorable graphs.} 
We include a concise proof here for completeness.
\begin{proof}
The set of functions $\wh \nu$ such that $\nu$ is supported on graphs whose complements are $(t-1)$-colorable forms a subspace of the space of functions $\F_2^N \to \R$. For $(t-1)$-colorable $T$, let $\nu_T$ be a point mass at its complement $\overline{T}$. Then $\wh \nu_T(\lambda) = (-1)^{e(\lambda)}(-1)^{e(\lambda \cap T)}$ lies in this subspace. The set of functions $\{\wh \nu_{T'}\}_{T' \subset T}$ forms a basis for the space of all functions of the form $(-1)^{e(\lambda)}f_T(T \cap \lambda)$, where $f_T: \F_2^N \to \R$ is arbitrary. The lemma follows by linearity.
\end{proof}

Although this gives us quite a large class of functions in which to search for an optimal $\wh\nu$, in practice we will only need the uniform distribution on complete $(t-1)$-partite graphs. In this case, $T$ may be viewed as a uniform random 3-coloring of $[n]$ and $T \cap \lambda$ is a random subgraph of $\lambda$ given by choosing a uniform random 3-coloring of its vertices and deleting all monochromatic edges.

\begin{definition}\label{def:G_q}
For a graph $G$ on $n$ labeled vertices, let $[q]^{V(G)}$ be the set of maps $\varphi \colon V(G) \to [q]$, viewed as $q$-colorings of $V(G)$ (not necessarily proper). 
For each coloring $\varphi \colon V(G) \to [q]$, define $G_\varphi$ to be the subgraph of $G$ formed by deleting all monochromatic edges of $G$, and then deleting all isolated vertices from the result. Let $G_q$ be the random graph $G_\varphi$ given by choosing $\varphi \sim \Unif([q]^{V(G)})$.
\end{definition}

\begin{proposition}\label{prop:construction} Let $t \in \{3,4\}$. There exists a set
of unlabeled graphs $\{H\}$, coefficients $\{c_H\}$ and $\delta > 0$ so that for any $G$ on $n$ labeled vertices, we have that
\begin{equation}
    \mu(G) := (-1)^{e(G)}\sum_{H} c_H \P[G_q \cong H]
\end{equation}
satisfies the following conditions.
\begin{enumerate}[(a)]
    \item $\mu(0) = 2^{\binom{t}{2}}-1$.
    \item $|\mu(G)| \le 1$ for all $G \neq \varnothing$
    \item $|\mu(G)| \le 1-\delta$ whenever $G$ has more than $\binom{t}{2}$ edges.
\end{enumerate} 
\end{proposition}

The constructions are presented in Section \ref{sec:constructions}, along with verification of Proposition \ref{prop:construction}. Assuming the proposition, we are nearly in a position to prove the main result, \cref{thm:main-agreeing}.
We first include two results from Boolean analysis.

\begin{lemma}[{\cite[Lemma 2.8]{Fri08}}]\label{lem:Friedgut}
Let $f:\F_2^N\to\{0,1\}$ be a monotone Boolean function with $\E[f] = 2^{-k}$ and $\wh f(S) = 0$ whenever $|S|> k$. Then $f$ is a $k$-umvirate.\footnote{Here $f$ being an $k$-umvirate means that there is some fixed set of $k$ coordinates so that $f$ is $1$ if and only if the input is $1$ on each of these $k$ coordinates.
\cite[Lemma 2.8]{Fri08} is stated for a biased distribution with $p < 1/2$, but with brief consideration it is easy to see the proof works for the unbiased case $p = 1/2$ as well, which is what we have stated here.}
\end{lemma}

\begin{theorem}[{\cite[Theorem 3]{KS02}}]\label{thm:K-S}
For every $k$, there exists $C > 0$ and $K$ such that for every $f: \F_2^N \to \{0,1\}$ there exists $g:\F_2^N \to \{0,1\}$ that depends on at most $K$ coordinates and satisfies
\begin{equation*}
    \P_{x \in \{0,1\}^N} [f(x) \neq g(x)] \le C\sum_{|S| > k} \wh f(S)^2.
\end{equation*}
\end{theorem}

To deduce uniqueness we will need to apply \cref{lem:Friedgut}, for which we require monotonicity. To that end, we employ a shifting argument. For a family of graphs $\mc F$ on $n$ labeled vertices and $e \in \binom{[n]}{2}$, the \textit{compression of $\mc F$ in direction $e$}, denoted $C_e(\mc F)$, is given by replacing each $G \in \mc F$ with $G \cup \{e\}$ whenever $e \notin E(G)$ and $G \cup \{e\}$ is not already in $\mc F$. It is easy to see from this definition that $|C_e(\mc F)| = |\mc F|$. 
\begin{lemma}\label{lem:shifting} Let $(\mc F_1, \mc F_2)$ be a cross-$t$-chromatic-agreeing pair of families of graphs on $n$ labeled vertices. Let $C = C_{e_1} \circ \cdots \circ C_{e_k}$ for some $k \in \Z$, $ e_1,\ldots e_k \in \binom{[n]}{2}$. Then $(C(\mc F_1), C(\mc F_2))$ is also cross-$t$-chromatic-agreeing. Moreover, if $C(\mc F_1) = C(\mc F_2)$ is a $K_t$-conjunction, then $\mc F_1 = \mc F_2$ is a $K_t$-conjunction as well.
\end{lemma}
\begin{proof}
By induction, it suffices to prove both claims in the case $C = C_{e}$. 
To check the first claim, note that if $G'_1 \in C_e(\mc F_1)$ and $G'_2 \in C_e(\mc F_2)$, then there exist $G_1 \in \mc F_1$ and $G_2 \in \mc F_2$ such that $G'_1$ and $G'_2$ agree wherever $G_1$ and $G_2$ agree (and possibly elsewhere as well).

The second claim is nearly identical to {\cite[Lemma 2.7]{EFF12}} (which was stated for an intersecting family rather than cross-intersecting families).
Fix some copy $T$ of $K_t$.
Assume $C_e(\mc F_1) = C_e(\mc F_2)$ is a $T$-conjunction. If $e$ is outside $T$, $C_e\inv$ must act trivially, and so $\mc F_1 = \mc F_2$ is the same conjunction. Otherwise, $e \in T$, and each choice of edges $H \subset \binom{[n]}{2}\sm E(T)$ extends uniquely to a graph in $\mc F_i$, for each $i$. View the set of such $H$ as a hypercube graph, with $H \sim H'$ if they differ in exactly one edge. Now consider the following two colorings of this hypercube. For $H$ in the hypercube, let $\chi_1(H)$ be given by membership of $e$ in the extension of $H$ to a member of $\mc F_1$, and let $\chi_2(H)$ be given by membership of $e$ in the extension of $K_n \setminus E(T \cup H)$ to a member of $\mc F_2$. If these two colorings $\chi_1$ and $\chi_2$ are constant and identical, then $\mc F_1$ and $\mc F_2$ are identical $T$-conjunctions. Otherwise, we can find a pair $(H, H')$ that disagree on at most one edge and satisfy $\chi_1(H) \neq \chi_2(H')$. Then the pair $(H, K_n \setminus E(T \cup H'))$ extends to $(G, G')  \in \mc F_1\times  \mc F_2$, whose agreement is contained in $K_t \sm \{e\} \cup \{\le 1\text{ other edge}\}$. This set of edges in agreement cannot possibly have chromatic number at least $t$, which completes the contradiction.
\end{proof}
\begin{corollary}\label{cor:monotone_suffices}
It suffices to prove (uniqueness) of \cref{thm:main-agreeing} in the case when $\mc F_1$ and $\mc F_2$ are monotone.
\end{corollary}
\begin{proof}
Whenever $C_e$ acts nontrivially on $\mc F_i$, it increases the total number of edges among graphs in $\mc F_i$ by at least 1. This quantity is uniformly upper bounded, and so after some finite number of applications $C = C_{e_1} \circ \cdots \circ C_{e_k}$ one has that $C(\mc F_1)$ and $C(\mc F_2)$ are invariant under all $C_e$, and hence monotone. By the first half of \cref{lem:shifting}, $C(\mc F_1)$ and $C(\mc F_2)$ are still $K_t$-chromatic-agreeing, so the monotone case of (uniqueness) tells us that they equal some $K_t$-conjunction. By the second half of \ref{lem:shifting}, we conclude that $\mc F_1$ and $\mc F_2$ are identical $K_t$-conjunctions as well.
\end{proof}

\subsection{Proof of \cref{thm:main-agreeing} assuming \cref{prop:construction}}
\begin{proof}[Proof (Upper bound)]
Let $(\mc F_1, \mc F_2)$ be $t$-chromatic agreeing.  Let $\nu:\F_2^N \to \R$ satisfy $\wh \nu = \mu/(2^{\binom{t}{2}}-1)$, where $\mu$ is the function given by Proposition \ref{prop:construction}. Then $\E[\nu] = \wh \nu(0) = 1$, and $m = \max_{\lambda \neq 0}|\wh \nu(\lambda)| = 1/(2^{\binom{t}{2}} - 1)$. By Lemmas \ref{lem:nu_support_condition} and \ref{lem:nu_construction_condition}, we see $\angles{f * g, \nu} = 0$. 
Applying Proposition \ref{prop:Hoffman}, we obtain $\E[f_1]\E[f_2] \le 2^{-2\binom{t}{2}}$.
\end{proof}

\begin{proof}[Proof (Uniqueness)] Taking $\mc F_1 = \mc F_2$ to be a $K_t$-conjunction achieves the upper bound. We need to show no other families may achieve the upper bound; by \cref{cor:monotone_suffices} it suffices to consider pairs $(\mc F_1, \mc F_2)$ that are monotone. By (maximal families) from Proposition \ref{prop:Hoffman} their indicator functions $\wh f_1, \wh f_2$ are supported on $\lambda$ where $|\wh \nu(\lambda)|$ is maximized, and they satisfy $\E[f_1] = \E[f_2]$.
Thus $\E[f_1]^2 = \E[f_2]^2 = \E[f_1]\E[f_2] = 2^{-2\binom{t}{2}}$. Since $\wh f_1, \wh f_2$ are supported only on $\lambda$ with $|\lambda| \le \binom{t}{2}$, we are in a position to apply Lemma \ref{lem:Friedgut}. This tells us that $f_1$ depends only on some set of coordinates $T_1$ and $f_2$ depends only on some set of coordinates $T_2$, where $|T_1| = |T_2| = \binom{t}{2}$. It is easy to choose $G_1 \in \mc F_1$ and $G_2 \in \mc F_2$ so that $G_1$ and $G_2$ disagree on all edges outside $T_1 \cap T_2$, so $T_1 \cap T_2$ is a graph on $\binom{t}{2}$ edges with chromatic number at least $t$. This can only happen when $T_1 = T_2 = T \cong K_t$, and now it is easy to see that to be cross-$t$-chromatic-agreeing $\mc F_1$ and $\mc F_2$ must be identical $T$-conjunctions.
\end{proof}

\begin{proof}[Proof (Stability)]
Applying (stability) from \cref{prop:Hoffman}, we see $\sum_{|S|> \binom{t}{2}} \wh f_i(S)^2 \le C\epsilon$ for some constant $C$ depending only on $t$ and $\delta$. Applying \cref{thm:K-S}, we see that each $f_i$ agrees with some $g_i$ on all but a $C'\epsilon$-fraction of inputs, where $g_i$ is a function only of some set $T_i$ of coordinates with $|T_i|\le K_0$. Let $V$ be the set of vertices incident to an edge of $E(T_1)$ or $E(T_2)$, so $|V| \le 4K_0$. Then $g_1$ and $g_2$ may be viewed as functions $\overline{g_1}$, $\overline{g_2}$ on the set of graphs on $V$, and in particular this restriction does not change expectation. Assuming for the sake of contradiction that $\overline {g_1}, \overline{g_2}$ are not identical $K_t$-conjunctions, we see by (uniqueness) that $\E[g_1]\E[g_2] < 2^{-2\binom{t}{2}}$. Since there are a finite number of non-maximal families on $4K_0$ vertices, we may write
\begin{equation*}
    \E[g_1]\E[g_2] \le \max_{\substack{h_1, h_2 \text{ nonmaximal}\\\text{ on $\le 4K_0$ vertices}}} \Big[\E[h_1]\E[h_2]\Big] =: 2^{-2\binom{t}{2}} - c, \quad\text{for some $c > 0$}.
\end{equation*}
If $\epsilon$ is sufficiently small that $2C'\epsilon+\epsilon < c$, then we have 
\begin{equation*}
    \E[f_1]\E[f_2] \le \E[g_1]\E[g_2]+2C'\epsilon \le 2^{-2\binom{t}{2}} - c + 2C'\epsilon < 2^{-2\binom{t}{2}} - \epsilon
\end{equation*}
which is a contradiction. Thus $\overline{g_1} = \overline{g_2}$ is the indicator function of some $K_t$-conjunction, and by extension so are $ f_1$ and $ f_2$.
\end{proof}

\section{How to verify the dual linear constraints}\label{sec:fourier}
In order to prove \cref{prop:construction}, once we have specified choices for $c_H$ we need to be able to effectively bound $\mu(G) = (-1)^{e(G)}\sum c_H \P[G_q \cong H]$ for all graphs $G$. (Recall from Definition \ref{def:G_q} that $G_q$ is a random subgraph of $G$ given by taking a random $q$-coloring of $V(G)$ and deleting all monochromatic edges, and then all isolated vertices.) The nature of $\mu$ makes this easy to do for sufficiently large $G$: by choosing $c_H$ to be supported on graphs with a bounded number of edges, it not too hard see that $\mu(G)$ must decay with $e(G)$. The proof is then a balance between optimizing control on the decay of $\mu(G)$ and doing casework on small graphs for which the generic bounds are not strong enough. The casework can thankfully be offloaded to a computer but computation time quickly becomes a limiting factor, as the number of graphs on $n$ vertices increases quickly with $n$. Our main tool to reduce to a finite computation is Proposition \ref{prop:reduction_to_bounded}.  We will begin with the derivation, and conclude with the statement of the bound. The derivation proceeds first by obtaining a result for connected graphs, and then bootstrapping that to all graphs.

The main idea of the proof is that if the coefficients $c_H$ are supported only on graphs of a bounded size, we can bound $\mu(G)$ by a contribution from each of its induced subgraphs on a bounded number of vertices. For a graph $G$ on $n$ labeled vertices, let $\kappa(G)$ be the number of connected components of $G$. We begin with a lemma.

\begin{lemma}\label{lem:restricted_colorings}
Let $q$ be a positive integer, and let $G'' \subseteq G' \subseteq G$ be labeled graphs, where the containments $V(G'') \subseteq V(G') \subseteq V(G)$ may be strict. Furthermore, assume $G$ is connected. Then 
\begin{equation*}
    \#\Big\{\varphi' \in [q]^{V(G')}: G'_{\varphi'} = G''\Big\} \ge q^{\kappa(G') - 1}\cdot \#\Big\{\varphi \in [q]^{V(G)}: G_{\varphi'} = G''\Big\}.
\end{equation*}

\end{lemma}
This lemma presents the somewhat counterintuitive fact that when the ambient graph $G$ is connected, a subgraph $G'$ may have more colorings which yield $G''$ than $G$ itself does. 
\begin{proof}
Let $\Phi_{G} := \{\varphi \in [q]^{V(G)}: G'_{\varphi'} = G''\}$, and $\Phi_{G'} := \{\varphi' \in [q]^{V(G')}: G'_{\varphi'} = G''\}$. Restriction of colorings from $V(G)$ to $V(G')$ gives a map $\rho: \Phi_{G} \to \Phi_{G'}$. First, we claim $\rho$ is injective. Let $\varphi \in \Phi_G$, so $G_\varphi = G'' \subset G'$. Then by definition, all edges in $E(G) \sm E(G')$ are monochromatic in $\varphi$. Since $G$ is connected, this means $\varphi$ is uniquely determined by its colors on $V(G')$, which proves the claim.

In fact, a more careful application of the same argument shows that the image of $\Phi_{G}$ under restriction $\rho$ is in fact much smaller than $\Phi_{G'}$. Let $\varphi' \in \Phi_{G'}$. Take some connected component of $G'$ and permute its colors via a cyclic shift, so $\varphi'(v) \mapsto \varphi'(v) + 1 \pmod{q}$ for $v$ in this component. After this operation, we still have $G'_{\varphi'} = G''$. This motivates the following definition. For $\varphi', \varphi'' \in \Phi_{G'}$, say $\varphi' \sim \varphi''$ if one may obtain $\varphi''$ by starting from the coloring $\varphi'$ and applying a sequence of cyclic shifts to the colors of connected components in $G'$. This is a well-defined equivalence relation, and each equivalence class has size $q^{\kappa(G')}$. We claim that the image $\rho(\Phi_G)$ contains at most $q$ elements of this equivalence class. As before, for any coloring $\varphi \in \Phi_G$, the edges in $E(G) \sm E(G')$ must be monochromatic. Since $G$ is connected, this means that $\varphi$ is uniquely determined once we specify the equivalence class of its restriction to $V(G')$ and the color of a single vertex $v_0$. Since there are $q$ choices for the color of $v_0$, the claim follows.

Concluding, $q|\Phi_{G'}| \le q^{\kappa(G')}|\Phi_{G}|$, which completes the proof.
\end{proof}
For a graph $G$ on $n$ labeled vertices, let $\binom{G}{n_0}$ be the set of labeled subgraphs of $G$ induced by all choices of $n_0$ vertices from $V(G)$.
\begin{lemma}\label{lem:double_count} Fix $q > 0$,  a list $\{H\}$ of unlabeled graphs on at most $n_0$ vertices and $\{c_H\}$ a list of coefficients. Then for any connected $G$ on $n \ge n_0$ labeled vertices,
\begin{equation}\label{eqn:double_count}
	\sum_H |c_H|\cdot \P[G_q \cong H] \leq 
	\sum_{G' \in \binom{G}{n_0}}\frac{1}{q^{\kappa(G') - 1}}\sum_{H} \frac{|c_{H}|}
         {q^{n-n_0}\binom{n-v(H)}{n_0-v(H)}}\P[ G'_q \cong H].
	\end{equation}
\end{lemma}
Although complicated, this bound is appealing for three reasons. First, when $\mu(G) = \sum_H c_H\cdot \P[G_\varphi \cong H]$, the LHS is an upper bound for $|\mu(G)|$ by the triangle inequality. Second, plugging in even the very crude bound $\sup |c_H|/q^{n-n_0}$ for the inner summation, this gives us the bound $|\mu(G)| \le \sup |c_H|\cdot n^{n_0}/q^{n-n_0}$ which decays exponentially in $n$. Third, each summand of the outer summation is a function only of $n$ and some graph $G'$ on $n_0$ vertices. Thus given $n$, we can obtain a bound on this quantity through direct computation, iterating through all colorings of all graphs on $n_0$ vertices. In fact, this bound will decay for $n$ sufficiently large, which will allow us to remove the $n$ dependence entirely and obtain a good uniform bound on $\mu$ for all graphs on $n > n_0$ vertices. 
\begin{proof}[Proof of Lemma \ref{lem:double_count}]
We employ a double-counting argument, summing Lemma \ref{lem:restricted_colorings} over all induced subgraphs $G' \subset G$ on $n_0$ vertices. For fixed $H$ with $c_H \neq 0$, we have
\begin{align*}
    |c_H|\cdot \P[G_q \cong H] &= \sum_{G \supseteq G'' \cong H} \frac{|c_{H}|}{q^n} \cdot \#\Big\{\varphi \in [q]^{V(G)}: G_{\varphi} = G''\Big\}\\
    &= \sum_{G' \in \binom{G}{n_0}}\sum_{G' \supseteq G'' \cong H} \frac{|c_{H}|}{q^n}\cdot
    \frac{\#\big\{\varphi \in [q]^{V(G)}: G_{\varphi} = G''\big\}}
         {\#\big\{G' \in \binom{G}{n_0}: G'' \subseteq G'\big\}}\\
    &\le \sum_{G' \in \binom{G}{n_0}}\sum_{G' \supseteq G'' \cong H} \frac{|c_{H}|}{q^n}\cdot
    \frac{q^{1- \kappa(G')}\cdot\#\big\{\varphi' \in [q]^{V(G')}: G'_{\varphi'} = G''\big\}}
         {\binom{n-v(G'')}{n_0-v(G'')}}\\
    &= \sum_{G' \in \binom{G}{n_0}}\frac{1}{q^{\kappa(G') - 1}}\frac{|c_{H}|}
         {q^{n-n_0}\binom{n-v(H)}{n_0-v(H)}}\P[ G'_q \cong H]
\end{align*}
Sum this inequality over all $H$ and we obtain the lemma.
\end{proof}

Now we eliminate the dependence of Lemma \ref{lem:double_count} on $n$. Take (\ref{eqn:double_count}) and upper bound the contents of the outer summation by taking its maximum over $G' \subset K_{n_0}$. The summation turns into a multiplicative factor of $\binom{n}{n_0}$, and so we may upper bound the RHS by
\begin{equation}\label{eqn:penultimate_mu_bound}
\max_{G' \subset K_{n_0}}\left[\frac{1}{q^{k(G') - 1}} \sum_{H} |c_H|\cdot \P[ G'_q \cong H]\cdot \frac{\binom{n}{n_0}}{q^{n-n_0} \binom{n - v(H)}{n_0 - v(H)}}\right].
\end{equation}
The final term is complicated, but decays with $n$, and can be replaced by a uniform upper bound. Let
\begin{equation}\label{eqn:DC_def}
	DC_{q, n_0}(x) := \max_{\substack{n \in \Z\\n > n_0}}\frac{\binom{n}{n_0}}{q^{n-n_0} \binom{n - x}{n_0 - x}}.
\end{equation}
The letters $DC$ stand for ``double counting.'' This expression changes by a multiplicative factor of $(n+1)/(q(n-x))$ when $n$ increases to $n+1$. So for $q \ge 2$ it is already decreasing for $n \ge 2x+1$, and it suffices to maximize for $n$ up to this bound. Below is a table of the values of $DC$ which covers our two constructions.  For all of these values, the maximum in (\ref{eqn:DC_def}) is achieved when $n = 10$, except for $DC_{3,9}(8)$ which is maximized at $n = 11$. 
\begin{figure}[H]
	\centering
	\begin{tabular}{ccccccccccc}
		$x$ & $0$ & $1$ & $2$ & $3$ & $4$ & $5$ & $6$ & $7$ & $8$   \\
		\midrule
		$DC_{2,9}(x)$ & $1/2$ & -- & $5/8$ & $5/7$ & $5/6$ & $1$ & -- & -- & --  \\
		$DC_{3,9}(x)$ & $1/3$ & -- & $5/12$ & $10/21$ & $5/9$ & $2/3$ & $5/6$ & $10/9$ & $55/27$ \\
	\end{tabular}
\end{figure}
Plugging DC into (\ref{eqn:penultimate_mu_bound}), we obtain a uniform bound on $|\mu(G)|$ for all connected $G$ with more than $n_0$ vertices.
\begin{lemma}\label{lem:reduction_to_bounded_connected}
Fix $q > 0$,  a list $\{H\}$ of unlabeled graphs on at most $n_0$ vertices and $\{c_H\}$ a list of coefficients. Then for any connected $G$ on $n > n_0$ labeled vertices, we have
\begin{equation}
	\sum_H |c_H|\cdot \P[G_q \cong H] \le \max_{G' \subset K_{n_0}}\left[\frac{1}{q^{k(G') - 1}} \sum_{H } |c_H|\cdot \P [G'_q \cong H]\cdot DC_{q, n_0}(v(H))\right].
\end{equation}
\end{lemma}

Now we extend this to a bound for all $G$. Let $G$ be a graph with at least two connected components, and let $u, v \in V(G)$ be a pair vertices in distinct connected components. Our first guess would be to replace $G$ by the graph obtained by identifying $u$ and $v$, which we denote $G/(u\sim v)$, and hope that $\sum_H |c_H|\cdot \P[G_\varphi \cong H]$ increases. This is nearly the case.
\begin{lemma}\label{lem:contraction_bound}
Fix $q > 0$,  a list $\{H\}$ of unlabeled graphs on at most $n_0$ vertices and $\{c_H\}$ a list of coefficients. Then for any $G$ on $n$ labeled vertices with at least two connected components, $u$ and $v$ a pair of disconnected vertices in $G$ (i.e., there is no path in $G$ from $u$ to $v$), and $G' = G/(u \sim v)$, we have
\begin{equation*}
    \sum_H |c_H|\cdot \P[G_q \cong H] \le \sum_H c'_H\cdot \P[G'_q \cong H], 
\end{equation*}
where $c'_H$ is the maximum value of $|c_{H'}|$ over all graphs $H'$ which may be transformed to $H$ by contracting at most one pair of disconnected vertices.
\end{lemma}
\begin{proof}
Let $\varphi$ be a random $q$-coloring of $V(G)$ distributed uniformly on colorings satisfying $\varphi(u) = \varphi(v)$. Since $u$ and $v$ are in distinct connected components, it is easy to see that $G_q$ and $G_{\varphi}$ have the same distribution. Any fixed $\varphi_0$ satisfying $\varphi_0(u) = \varphi_0(v)$ descends to a well-defined coloring $\varphi_1$ on $G' = G/(u \sim v)$. Furthermore, we may get from $G_{\varphi_0}$ to $G'_{\varphi_1}$ by identifying at most one pair of disconnected vertices. Consequently, the following inequality holds term by term, equating terms under this map $\varphi_0 \mapsto \varphi_1$.
\begin{align*}
    \sum_H |c_H|\cdot \P[G_q \cong H] &= \frac{1}{q^{n-1}}\sum_{\substack{\varphi_0  \in [q]^{V(G)}\\ \varphi_0(u) = \varphi_0(v)}} |c_{G_{\varphi_0}}|\\
    &\le \frac{1}{q^{n-1}}\sum_{\varphi_1  \in [q]^{V(G')}} {c}'_{G'_{\varphi_1}}\\
    &= \sum_H {c'_H}\cdot \P[G'_q \cong H].
\end{align*}\end{proof}
Iteratively applying Lemma \ref{lem:contraction_bound} until $G$ is connected, and then applying Lemma \ref{lem:reduction_to_bounded_connected}, we obtain the following result, which allows us to upper bound $\mu$ by a finite computation.
\begin{proposition}\label{prop:reduction_to_bounded}
Fix $q > 0$, a list $\{H\}$ of unlabeled graphs on at most $n_0$ vertices and $\{c_H\}$ a list of coefficients. Then for any $G$ on $n > n_0$ labeled vertices, we have
\begin{equation}\label{eqn:reduction_to_bounded}
	\sum_H |c_H|\cdot \P[G_q \cong H] \le \max_{G' \subset K_{n_0}}\left[\frac{1}{q^{k(G') - 1}} \sum_{H} \widetilde{c_H}\cdot \P[G'_q \cong H]\cdot DC_{q, n_0}(v(H))\right],
\end{equation}
where $\widetilde{c_H}$ is the maximum value of $|c_{H'}|$ over all graphs $H'$ which may be transformed to $H$ by repeatedly identifying pairs of disconnected vertices.
\end{proposition}

\section{Selecting a feasible dual solution}\label{sec:constructions}
We now prove Proposition \ref{prop:construction} by computer verification.
Computations were scripted in Python 3.6.5 using integer arithmetic. We use the list of all graphs up to nine vertices, taken from \cite{McK}. Our code is attached as ancillary files to the arXiv version of this article.
\subsection{The case $t = 3$} The choices of $c_H$ are listed in \cref{tab:k3}. It turns out that in order to satisfy the conclusions of \cref{prop:construction}, the values of $c_H$ for $H \subseteq K_3$ are uniquely determined. The other coefficients were chosen with some flexibility. 

\begin{table}[H]
	\centering
	\begin{tabular}{cccccccccc}
		$H$ & $\varnothing$ & \tikzGraph{.3}{\tikzE} & \tikzGraph{.3}{\tikzPtwo} & \tikzGraph{.3}{\tikzEE} & \tikzGraph{.1732}{\tikzSthree} & \tikzGraph{.3}{\tikzPthree} & \tikzGraph{.1732}{\tikzSfour} & \tikzGraph{.3}{\tikzCfour} & \tikzGraph{.3}{\tikzPfour} \\
	    \midrule
		$c_H$ & $7$ & $-5$ & $-1$ & $1.7$ & $3$ & $0.3$ & $-0.2$ & $-3.7$ & $-0.75$\\
	\end{tabular}
	\caption{Coefficients for $t = 3$.}\label{tab:k3}
\end{table}
With this choice of $c_H$, we verify Proposition \ref{prop:construction} for $t = 3$.
\begin{lemma} With $c_H$ chosen as in \cref{tab:k3}, and $\mu(G) = (-1)^{e(G)}\sum_{H} c_H \P[G_q \cong H]$, one has
\begin{enumerate}
    \item $\mu(0) = 7$.
    \item $|\mu(G)| \le 1$ whenever $G$ has at most $3$ edges.
    \item $|\mu(G)| \le 0.9875$ whenever $G$ has more than $3$ edges.
\end{enumerate}
\end{lemma}
\begin{proof}
For the empty graph, $G_q$ is always empty, and so $\mu(0) = c_\varnothing = 7$. Properties (2) and (3) are verified by a finite computation, iterating through all 2-colorings of all graphs with up to 9 vertices. Property (3) is verified for graphs with more than 9 vertices by applying Proposition \ref{prop:reduction_to_bounded} with $n_0 = 9$, again by iterating through all 2-colorings of all graphs with up to 9 vertices.
\end{proof}

\subsection{The case $t = 4$}
Our choices for the coefficients $c_H$ are listed in \cref{tab:K4}. For $t = 4$ one needs to choose a substantially larger number of coefficients to be nonzero. Despite this, there is still considerable flexibility. To make a reasonably concise presentation, we choose our list of coefficients $\{c_H\}$ to be constant on certain equivalence classes of $H$, so that we may simply write down a single choice for each class. The classes are defined as follows:
\begin{definition}
A \textit{block} of a graph $H$ is a maximal connected subgraph with at least one edge and no cut vertex. The collection of blocks of $H$ partitions $E(H)$. We say two graphs $H$ and $H'$ are equivalent and write $H \sim H'$ if the collection of blocks of $H$ and the collection of blocks of $H'$ are equal as multisets of unlabeled graphs.
\end{definition}

\begin{example}
The graphs \tikzGraph{.3}{\tikzPtwo} and \tikzGraph{.3}{\tikzEE} are equiavlent, with blocks \tikzGraph{.3}{\tikzE}, \tikzGraph{.3}{\tikzE}.
\end{example}
\begin{example}
The graph \tikzGraph{.15}{\tikzKthreeKthree} has blocks \tikzGraph{.15}{\tikzKthree},\tikzGraph{.15}{\tikzKthree}.
\end{example}

We will choose $c_H$ to be equal for all graphs within each equivalence class. This additional restriction is motivated in two ways. First, all graphs $H'$ involved in the computation of $\widetilde{c_H}$ (see Proposition \ref{prop:reduction_to_bounded}) satisfy $H' \sim H$, and so with this choice one has $\widetilde{c_H} = |c_H|$ for all $H$. Second, it turns out that for general $t$, if $H \sim H'$ and both are contained in a $K_t$, in order to satisfy the conclusions of \cref{prop:construction} one must choose $c_H = c_{H'}$ regardless. 
\newcolumntype{d}[1]{D{.}{.}{#1}}
\begin{table}[H]
	\centering
	\begin{tabular}{ld{3.3}}
		Blocks of $H$  &   \multicolumn{1}{c}{$c_H$}\\
		\midrule
		$\varnothing$ & 63\\
		\tikzGraph{.3}{\tikzE} & -30\\
		\tikzGraph{.3}{\tikzE},\tikzGraph{.3}{\tikzE} & 12\\
		\tikzGraph{.1732}{\tikzKthree} & -63\\
		\tikzGraph{.1732}{\tikzKthree},\tikzGraph{.3}{\tikzE} & 6 \\
		\tikzGraph{.3}{\tikzCfour} & -39\\
		\tikzGraph{.3}{\tikzE} ,\tikzGraph{.3}{\tikzE} ,\tikzGraph{.3}{\tikzE} ,\tikzGraph{.3}{\tikzE} & -1.926\\
		\tikzGraph{.3}{\tikzDiamond} & 12\\
		\tikzGraph{.1732}{\tikzKthree}, \tikzGraph{.3}{\tikzE} ,\tikzGraph{.3}{\tikzE} & 5.478\\
		\tikzGraph{.3}{\tikzCfour}, \tikzGraph{.3}{\tikzE} &4.293\\
		\tikzGraph{.3}{\tikzCfive} & -16.5\\
	\end{tabular}
	\quad
	\begin{tabular}{ld{3.3}}
		Blocks of $H$  &   \multicolumn{1}{c}{$c_H$}\\
		\midrule
		
		\tikzGraph{.3}{\tikzThetathree} & -20.3\\
		\tikzGraph{.3}{\tikzKthreetwo} & 24.75\\
		\tikzGraph{.3}{\tikzCfour},\tikzGraph{.3}{\tikzE} ,\tikzGraph{.3}{\tikzE} & -2.274 \\
		\tikzGraph{.3}{\tikzCfive}, \tikzGraph{.3}{\tikzE} & 2\\
		\tikzGraph{.3}{\tikzCsix} & -11.528\\
		\tikzGraph{.3}{\tikzKthreeoneone} & -22.928\\
		\tikzGraph{.1732}{\tikzTrap} & -12.8\\
		\tikzGraph{.3}{\tikzThetathree}, \tikzGraph{.3}{\tikzE} &-1.2\\
		\tikzGraph{.3}{\tikzKthreetwo}, \tikzGraph{.3}{\tikzE} &-4.362\\
		\tikzGraph{.3}{\tikzThetafour} & 1.138\\
		~&~\\
	\end{tabular}
	\caption{Coefficients for $t = 4$}\label{tab:K4}
\end{table}

\begin{lemma} With $c_H$ chosen as in \cref{tab:k3}, and $\mu(G) = (-1)^{e(G)}\sum_{H} c_H \P[G_q \cong H]$, one has
\begin{enumerate}
    \item $\mu(0) = 63$,
    \item $|\mu(G)| \le 1$ whenever $G$ has at most $6$ edges, and
    \item $|\mu(G)| \le 0.999$ whenever $G$ has more than $6$ edges.
\end{enumerate}
\end{lemma}
\begin{proof}
For the empty graph, $G_q$ is always empty, and so $\mu(0) = c_\varnothing = 63$. Properties (2) and (3) are verified by a finite computation, iterating through all 3-colorings of all graphs with up to 9 vertices. On graphs with at most 9 vertices and more than 6 edges, the maximum value of $|\mu(G)|$ is $1 - 188406/(1000 \cdot 3^9)$, which is attained by \tikzGraph{.3}{\tikzExtremal}. Property (3) is verified for graphs with more than 9 vertices by applying Proposition \ref{prop:reduction_to_bounded} with $n_0 = 9$, by iterating through all 3-colorings of all graphs with up to 9 vertices, which gives a bound of $1-20/19683$.
\end{proof}

\section{Additional proofs} \label{sec:additional-proofs}
Here we collect the proofs deducing \cref{thm:main-intersecting} and \cref{cor:smaller_p} from \cref{thm:main-agreeing}. 
The following fact is well known (see \cite[Theorem 2.38]{Gri99} or \cite[Lemma 2.6]{EKL19}).
\begin{lemma}\label{lem:monotone}
Let $\mc F \subset 2^{[n]}$ be a nonempty monotone increasing family (that is, for all $X \in \mc F$ and $Y \supseteq X$, $Y \in \mc F$). For $0 < p < 1$ let $\mu_p(\mc F) = \sum_{X \in \mc F} p^{|X|}(1-p)^{n-|X|}$. Then $\log_p \mu_p(\mc F)$ is non-increasing in $p$.
\end{lemma}
\begin{corollary}\label{cor:cross-intersecting}
Let $0 < p \le 1/2$ and let $({\mc F_1},{\mc F_2})$ be a cross-intersecting pair of subsets of $[N]$. Then $\mu_p(\mc F_1)\mu_p(\mc F_2) \le p^2$. 
\end{corollary}
\begin{proof}
It suffices to prove this bound when ${\mc F_1}$ and ${\mc F_2}$ are replaced by their upwards closures (the upwards closure of $\mc F$ is the set $\{H: \exists G \in \mc F: G \subseteq H\}$).  By \cref{lem:monotone} it suffices to consider the $p = 1/2$ case. In this setting, consider complementary pairs of sets $(G,\overline G)$. For $({\mc F_1},{\mc F_2})$ to be an intersecting pair of families, we must have $\ind_{G \in \mc F_1} + \ind_{\overline G \in \mc F_2} \le 1$. Taking the expectation over all $G$, by linearity of expectation this becomes $\mu_{1/2}(\mc F_1)+\mu_{1/2}(\mc F_2) \le 1$, and so by AM-GM $\mu_{1/2}(\mc F_1)\mu_{1/2}(\mc F_2) \le 1/4$, which completes the proof.
\end{proof}
We continue with a lemma that allows us to strengthen the conclusions of \cref{thm:main-agreeing} from $K_t$-conjunctions to $K_t$-umvirates under the stronger hypothesis that $\mc F_1, \mc F_2$ are cross-$t$-chromatic-intersecting. Recall that $\mu_p(\mc F) := \sum_{G \in \mc F} p^{e(G)}(1-p)^{e(\overline G)}.$ For a set of edges $E_0 \subset \binom{[n]}{2}$, we say a family of graphs $\mc F$ is an \textit{$E_0$-conjunction} if there is some graph $H$ with $E(H) \subset E_0$ so that $\mc F = \{G: E_0 \subseteq E(G \triangle H)\}$. If $|E_0| = k$ we say $\mc F$ is a \textit{$k$-conjunction}.
\begin{lemma}\label{lem:must_be_umvirate}
Let $t \in \N$ and let $0 < p \le 1/2$. There exists a constant $\epsilon_{p,t} > 0$ so that the following holds. Let $(\mc F_1, \mc F_2)$ be a cross-$t$-chromatic-intersecting pair of families of graphs on $n$ labeled vertices, and let $\mc T \subset 2^{\binom{[n]}{2}}\times 2^{\binom{[n]}{2}}$ be a $k$-conjunction for $k \le 2\binom{t}{2}$. If $\mc T$ is not the Cartesian product of identical $K_t$-umvirates, then $\mu_p((\mc F_1 \times \mc F_2) \cap \mc T) \le (1-\epsilon_{p,t})\mu_p(\mc T)$.
\end{lemma}
\begin{proof}
Any conjunction $\mc T$ admits a product decomposition $\mc T = \mc T_1 \times \mc T_2$, where $\mc T_i:2^{\binom{[n]}{2}} \to \{0,1\}$. We can write $\mc T_i = \{G: E_i \subseteq E(G \triangle H_i)\}$ for some graphs $H_i$ and sets of edges $E_i$ where $|E_1| + |E_2| \le 2\binom{t}{2}$ . If we are not in the case where $E_1 = E_2 \cong K_t$ and $H_1 = H_2 = \emptyset$, then for any $G_1 \in \mc T_1$, $G_2 \in \mc T_2$, we see $(G_1 \cap G_2) \cap (E_1 \cap E_2)$ has at most $\binom{t}{2}$ edges and is not a $K_t$, and therefore must be $(t-1)$-colorable, and so $G_1$ and $G_2$ must share an edge outside $E_1 \cap E_2$.

Now we construct subsets of each conjunction that ``disagree with the other conjunction as much as possible.'' Consider $E_2 \sm E_1$. Graphs in $\mc T_2$ must disagree with $H_2$ on this set, but graphs in $\mc T_1$ are unconstrained on this set of edges. Consider the subset of $\mc T_1$ given by graphs that agree with $H_2$ on this set:  $\mc T_1' := \{G \in \mc T_1: (G \triangle H_2) \cap (E_2 \sm E_1) = \emptyset\}$. Define $\mc T_2'$ similarly. Then since $|E_2 \sm E_1| \le 2\binom{t}{2}$, we see $\mu_p(\mc T_i') \ge p^{2\binom{t}{2}}\mu_p(\mc T)$, and each pair $G_1 \in \mc F_1 \cap \mc T_1'$, $G_2 \in \mc F_2 \cap \mc T_2'$ intersects on no edges in $E_1 \triangle E_2$. We know that $G_1$ and $G_2$ must share an edge outside $E_1 \cap E_2$, which must therefore lie outside $E_1\cup E_2$. Letting $\overline{\mc F_i} = \{E(G) \sm (E_1 \cup E_2): G \in \mc F_i \cap T_i'\}$ (removing duplicates), we have that $(\overline{\mc F_1},\overline{\mc F_2})$ is a cross-intersecting pair of families.  Restricting our domain to $\binom{[n]}{2} \sm (E_1 \cup E_2)$, by \cref{cor:cross-intersecting} we have $\mu_p(\overline{\mc F_1})\mu_p(\overline{\mc F_2}) < p^2$.  Pulling this back, we see $\mu_p(\mc F \cap (\mc T_1'\times \mc T_2')) \le p^2\mu_p(\mc T_1'\times \mc T_2')$.

To conclude, write $\mu_p(\mc F\cap \mc T) = \mu_p((\mc F \cap \mc T)\sm \mc (T_1'\times \mc T_2')) + \mu_p(\mc F \cap (\mc T_1'\times \mc T_2'))$. The first term is trivially bounded by $\mu_p(\mc T\sm  (\mc T_1'\times \mc T_2'))$, and we just saw that the second term is at most $p^2\mu_p(\mc T_1'\times \mc T_2')$. Adding in our lower bound $\mu_p(\mc T_i) \ge p^{-2\binom{t}{2}}$, we conclude $\mu_p(\mc F\cap \mc T) \le (1-(1-p^2)p^{4\binom{t}{2}}) \mu_p(\mc T)$,
which completes the proof.
\end{proof}
We can apply this to immediately deduce \cref{thm:main-intersecting} from \cref{thm:main-agreeing}.
\begin{lemma}\label{lem:main_intersecting}
Let $(\mc F_1, \mc F_2)$ be a cross-$t$-chromatic-intersecting pair of families of graphs on $n$ labeled vertices. Then equality holds in Theorem \ref{thm:main-agreeing} if and only if $\mc F_1 = \mc F_2$ is a $K_t$-umvirate, and for sufficiently small $\epsilon$ the $K_t$-conjunction $\mc T$ from (stability) must be a $K_t$-umvirate.
\end{lemma}
\begin{proof}
The upper bound and uniqueness follow immediately by noting that the only $K_t$-conjunction that is $K_t$-intersecting is a $K_t$-umvirate. For stability, if \cref{thm:main-agreeing} outputs some $\mc T$ that is not a $K_t$-umvirate, we can apply \cref{lem:must_be_umvirate} to $\mc T \times \mc T$ with $p = 1/2$ and deduce $|\mc F_1\cap \mc T||\mc F_2\cap \mc T| \le (1-c_t)|\mc T|^2.$ But then we cannot have satisfied the hypothesis of \cref{thm:main-agreeing} (stability) for sufficiently small $\epsilon$.
\end{proof}
\begin{proof}[Proof of \cref{thm:main-intersecting}]
\cref{lem:main_intersecting} completes the proof of \cref{thm:main-intersecting} in the more general cross-$t$-chromatic-intersecting setting, for sufficiently small $\epsilon$. Choosing a new $C_t$ large enough makes the conclusion of (stability) trivially satisfied for large $\epsilon$, which completes the deduction.
\end{proof}

Now that we have handled the $p=1/2$ case for cross-$t$-chromatic-intersecting families, we next apply the results of \cite{EKL19} to obtain results for all $p < 1/2$ to yield \cref{cor:smaller_p}.

\begin{corollary}
    Let $t \in \{3,4\}$, $0 < p < 1/2$, and let $(\mc F_1, \mc F_2)$ be a cross-$t$-chromatic-intersecting pair of families of graphs on $n$ labeled vertices. Then $\mu_p(\mc F_1)\mu_p(\mc F_2) \le p^{2\binom{t}{2}}$, with equality if and only if $\mc F_1 = \mc F_2$ is a $K_t$-umvirate.
\end{corollary}
\begin{proof}
The upwards closures of $\mc F_1$ and $\mc F_2$ also form a cross-$t$-chromatic-intersecting pair of families, so for the original pair to be maximal both families be upwards-closed. The result then follows from the $p=1/2$ version above and \cref{lem:monotone}.
\end{proof}

In order to deduce stability we apply the following powerful result from \cite{EKL19}.

\begin{theorem}[{\cite[Theorem 3.1]{EKL19}}]\label{thm:EKL}
Let $t\in \N$ and $0 < p < 1/2$. Then exist constants $C > 4$ and $c > 0$ so that the following holds. Let $f : \F_2^N \to \{0,1\}$ be increasing with $\E[f] \le 2^{-t}$ and
\begin{equation*}
    \mu_p(\mc F) \ge \begin{cases}
    p^t\Big(1-c\Big(\frac 12-p\Big)\Big)& p \ge 1/C\\
    Cp^{t+1} & p < 1/C.
    \end{cases}
\end{equation*}
For all $\epsilon > 0$, if 
\begin{equation*}
    \mu_p(f) \ge p^t\Big(1-\epsilon^{\log_p(1-p)}\Big)+p^{t-1}(1-p)\epsilon,
\end{equation*}
then there is some $t$-umvirate $\mc G$ such that $\mu_p(\mc F \sm \mc G) \le (1-p)p^{t-1}\epsilon.$
\end{theorem}
\begin{corollary}
Let $t \in \N$ and let $0 < p < 1/2$. Then there exists a constant $C_{p,t} > 0$ so that the following holds. Let $(\mc F_1, \mc F_2)$ be a cross-$t$-chromatic-intersecting pair of families of graphs on $n$ labeled vertices. Then if $\mu_p(\mc F)\ge (1-\epsilon)p^{\binom{t}{2}}$, then there exists some $K_t$-umvirate $\mc T$ so that $\mu_p(\mc F_i \sm \mc T) \le C_{p,t}\epsilon^{\log_{(1-p)}(p)}$ for $i \in \{1,2\}$.
\end{corollary}

\begin{proof}
By Lemma \ref{lem:main_intersecting}, $\mc F = \mc F_1 \times \mc F_2$ has size $\mu_{1/2}(\mc F)\le 2^{-2\binom{t}{2}}$. Then by \cref{thm:EKL} if $\mu_p(\mc F) \ge (1-\epsilon)p^{2\binom{t}{2}}$, there is some $2\binom{t}{2}$-umvirate $\mc T$ so that $\mu_p(\mc F \sm \mc T) \le C'_{p,t}\epsilon^{\log_{(1-p)}(p)}$. If $\mc T$ is not the Cartesian product of identical $K_t$-umvirates, apply \cref{lem:must_be_umvirate}. Then $\mu_p(\mc F_1)\mu_p(\mc F_2) = \mu_p(\mc F \cap \mc T)+\mu_p(\mc F \sm \mc T) \le (1-\epsilon_{p,t})\mu_p(\mc T) + C'_{p,t}\epsilon^{\log_{(1-p)}(p)} < (1-\epsilon)p^{2\binom{t}{2}}$ for $\epsilon$ sufficiently small. Choosing $C_{p,t}$ large enough makes the conclusion trivial for large $\epsilon$, which completes the proof.
\end{proof}

\section{Further work} \label{sec:discussion}
Given these results, we are optimistic that this framework is strong enough to obtain optimal bounds on $K_t$-intersecting families for each fixed $t$ (given enough computational power). Despite this, proving a result for all $t$ will require some new insight. We see two potential directions to proving a result in full generality. The first is to bypass the $c_H$ setup entirely by writing down a $\mu$ which satisfies the conclusions of \cref{prop:construction}.
We see no obvious candidates for such $\mu$, and the constructions of \cite{EFF12} and this paper provide no clear pattern for generalization. The second option would be to write down a general form for the coefficients $c_H$, and use some technique akin to \cref{prop:reduction_to_bounded} to reduce checking \cref{prop:construction} to a bounded computation that can be done by hand. This also seems difficult, as again there is no clear pattern for generalization among the feasible solutions for $t = 3,4$. 

The \textit{tetrahedron} is the complete 3-uniform hypergraph on 4 vertices. In correspondence, Ellis, Filmus, and Friedgut suggested that the following version of \cref{prob:intersecting} would be interesting.
\begin{problem}
What is the maximal size of a tetrahedron-intersecting family of 3-uniform hypergraphs on $n$ labeled vertices?
\end{problem}

\noindent
\bibliographystyle{amsplain0}

\bibliography{intersecting}

\end{document}

%% file: mypreamble.tex
\usepackage{amsmath, amssymb, amsthm}
\usepackage[margin=1in]{geometry}
\usepackage{verbatim, dsfont}
\usepackage{graphicx}
\usepackage[shortlabels]{enumitem}
\usepackage{listings}
\usepackage[ruled,vlined,linesnumbered]{algorithm2e}

\usepackage{xcolor, float}
\usepackage[colorlinks=true,allcolors=blue]{hyperref}

\usepackage[noabbrev,capitalize]{cleveref}
\crefname{equation}{}{}

\newcommand{\angles}[1]{\left\langle #1 \right\rangle}
\newcommand{\bb}{\mathbb}

\newcommand{\E}{\mathop{\bb E}}
\newcommand{\F}{\bb F}

\newcommand{\ind}{\mathds{1}}
\newcommand{\inv}{^{-1}}

\newcommand{\mc}{\mathcal}

\newcommand{\sm}{\setminus}
\newcommand{\R}{\bb R}

\newcommand{\Z}{\bb Z}
\renewcommand{\P}{\bb P}

\newtheorem{theorem}{Theorem}[section]
\newtheorem{proposition}[theorem]{Proposition}
\newtheorem{lemma}[theorem]{Lemma}

\newtheorem{corollary}[theorem]{Corollary}
\newtheorem{conjecture}[theorem]{Conjecture}

\theoremstyle{definition}
\newtheorem{definition}[theorem]{Definition}
\newtheorem{problem}[theorem]{Problem}

\newtheorem{example}[theorem]{Example}

\theoremstyle{remark}

\newcommand{\abs}[1]{\left\lvert#1\right\rvert}

\newcommand{\wh}{\widehat}

\newcommand{\ol}{\overline}

\newcommand{\tikzGraph}[2]{
\begin{tikzpicture}[every node/.style={circle, fill, inner sep = 1pt},scale={#1}]
#2
\end{tikzpicture}
}

\newcommand{\tikzPtwo}{
\node (a) at (0,0) {};
\node (b) at (1,1) {};
\node (c) at (2,0) {};
\draw (a)--(b)--(c);
}
\newcommand{\tikzEE}{
\node (a) at (0,0) {};
\node (b) at (1,0) {};
\node (c) at (0,1) {};
\node (d) at (1,1) {};
\draw (a)--(b);
\draw (c)--(d);
}
\newcommand{\tikzE}{
\node (a) at (0,0) {};
\node (b) at (1,1) {};
\draw (a)--(b);
}
\newcommand{\tikzKthree}{
\node (a) at (-1,0) {};
\node (b) at (0,1.732) {};
\node (c) at (1,0) {};
\draw (c)--(b)--(a)--(c);
}
\newcommand{\tikzKthreeKthree}{
\node (a) at (0,0) {};
\node (b) at (1,1.732) {};
\node (c) at (2,0) {};
\node (d) at (3,1.732) {};
\node (e) at (4,0) {};
\draw (c)--(b)--(a)--(c);
\draw (c)--(d)--(e)--(c);
}
\newcommand{\tikzCfour}{
\node (a) at (0,0) {};
\node (b) at (1,0) {};
\node (c) at (1,1) {};
\node (d) at (0,1) {};
\draw (a)--(b)--(c)--(d)--(a);
}
\newcommand{\tikzDiamond}{
\node (a) at (0,0) {};
\node (b) at (1,0) {};
\node (c) at (1,1) {};
\node (d) at (0,1) {};
\draw (a)--(b)--(c)--(d)--(a)--(c);
}
\newcommand{\tikzCfive}{
\node (a) at (0,0) {};
\node (b) at (1,0) {};
\node (c) at (2,.5) {};
\node (d) at (1,1) {};
\node (e) at (0,1) {};
\draw (a)--(b)--(c)--(d)--(e)--(a);
}
\newcommand{\tikzThetathree}{
\node (a) at (0,0) {};
\node (b) at (1,0) {};
\node (c) at (2,.5) {};
\node (d) at (1,1) {};
\node (e) at (0,1) {};
\draw (a)--(b)--(c)--(d)--(e)--(a);
\draw (d)--(b);
}
\newcommand{\tikzCsix}{
\node (a) at (0,0) {};
\node (b) at (1,0) {};
\node (c) at (2,0) {};
\node (d) at (2,1) {};
\node (e) at (1,1) {};
\node (f) at (0,1) {};
\draw (a)--(b)--(c)--(d)--(e)--(f)--(a);
}
\newcommand{\tikzKthreeoneone}{
\node (a) at (0,0) {};
\node (b) at (2,0) {};
\node (c) at (2,1) {};
\node (d) at (1,1) {};
\node (e) at (0,1) {};
\draw (c)--(a)--(d)--(a)--(e);
\draw (c)--(b)--(d)--(b)--(e);
\draw (a)--(b);
}
\newcommand{\tikzTrap}{
\node (a) at (0,0) {};
\node (b) at (1,1.732) {};
\node (c) at (2,0) {};
\node (d) at (3,1.732) {};
\node (e) at (4,0) {};
\draw (c)--(b)--(a)--(c);
\draw (c)--(d)--(e)--(c);
\draw (b)--(d);
}
\newcommand{\tikzThetafour}{
\node (a) at (0,0) {};
\node (b) at (1,0) {};
\node (c) at (2,0) {};
\node (d) at (2,1) {};
\node (e) at (1,1) {};
\node (f) at (0,1) {};
\draw (a)--(b)--(c)--(d)--(e)--(f)--(a)--(e);
}
\newcommand{\tikzKthreetwo}{
\node (a) at (0,0) {};
\node (b) at (2,0) {};
\node (c) at (2,1) {};
\node (d) at (1,1) {};
\node (e) at (0,1) {};
\draw (c)--(a)--(d)--(a)--(e);
\draw (c)--(b)--(d)--(b)--(e);
}
\newcommand{\tikzSthree}{
\node (b) at (1,1.732) {};
\node (c) at (2,0) {};
\node (d) at (3,1.732) {};
\node (e) at (4,0) {};
\draw (b)--(c)--(d)--(c)--(e);
}
\newcommand{\tikzSfour}{
\node (a) at (0,0) {};
\node (b) at (1,1.732) {};
\node (c) at (2,0) {};
\node (d) at (3,1.732) {};
\node (e) at (4,0) {};
\draw (a)--(c)--(b)--(c)--(d)--(c)--(e);
}

\newcommand{\tikzPthree}{
\node (a) at (0,0) {};
\node (b) at (.5,1) {};
\node (c) at (1,0) {};
\node (d) at (1.5,1) {};
\draw (a)--(b)--(c)--(d);
}

\newcommand{\tikzPfour}{
\node (a) at (0,0) {};
\node (b) at (.5,1) {};
\node (c) at (1,0) {};
\node (d) at (1.5,1) {};
\node (e) at (2,0) {};
\draw (a)--(b)--(c)--(d)--(e);
}
\newcommand{\tikzExtremal}{
\node (a) at (0,0) {};
\node (b) at (1,0) {};
\node (c) at (2,0) {};
\node (d) at (2,1) {};
\node (e) at (1,1) {};
\node (f) at (0,1) {};
\draw (a)--(b)--(c)--(d)--(e)--(f)--(a)--(e)--(b);
}

%% file: main.bbl
\providecommand{\bysame}{\leavevmode\hbox to3em{\hrulefill}\thinspace}
\providecommand{\MR}{\relax\ifhmode\unskip\space\fi MR }
\providecommand{\MRhref}[2]{%
  \href{http://www.ams.org/mathscinet-getitem?mr=#1}{#2}
}
\providecommand{\href}[2]{#2}
\begin{thebibliography}{10}

\bibitem{BBDDDMV19}
Aaron Berger, Ross Berkowitz, Pat Devlin, Michael Doppelt, Sonali Durham, Tessa
  Murthy, and Harish Vemuri, \emph{Connected-intersecting families of graphs},
  \href{https://arxiv.org/abs/1901.01616}{arXiv:1901.01616}.

\bibitem{CGFS86}
F.~R.~K. Chung, R.~L. Graham, P.~Frankl, and J.~B. Shearer, \emph{Some
  intersection theorems for ordered sets and graphs}, J. Combin. Theory Ser. A
  \textbf{43} (1986), 23--37.

\bibitem{Del73}
Philippe Delsarte, \emph{An algebraic approach to the association schemes of
  coding theory}, Philips Res. Rep. Suppl. (1973), vi+97.

\bibitem{EFF12}
David Ellis, Yuval Filmus, and Ehud Friedgut, \emph{Triangle-intersecting
  families of graphs}, J. Eur. Math. Soc. (JEMS) \textbf{14} (2012), 841--885.

\bibitem{EKL19}
David Ellis, Nathan Keller, and Noam Lifshitz, \emph{Stability versions of
  {Erd\H{o}s-Ko-Rado} type theorems via isoperimetry}, J. Eur. Math. Soc.
  (JEMS) \textbf{21} (2019), 3857--3902.

\bibitem{EKR61}
Paul Erd\H{o}s, Chao Ko, and Richard Rado, \emph{Intersection theorems for
  systems of finite sets}, Quart. J. Math. Oxford Ser. (2) \textbf{12} (1961),
  313--320.

\bibitem{Fri08}
Ehud Friedgut, \emph{On the measure of intersecting families, uniqueness and
  stability}, Combinatorica \textbf{28} (2008), 503--528.

\bibitem{Gri99}
Geoffrey Grimmett, \emph{Percolation}, second ed., Grundlehren der
  Mathematischen Wissenschaften [Fundamental Principles of Mathematical
  Sciences], vol. 321, Springer-Verlag, Berlin, 1999.

\bibitem{KL19}
Nathan Keller and Noam Lifshitz, \emph{A note on large {$H$}-intersecting
  families}, SIAM J. Discrete Math. \textbf{33} (2019), 398--401.

\bibitem{KS02}
Guy Kindler and Shmuel Safra, \emph{Noise-resistant boolean functions are
  juntas}, unpublished (2004), \url{https://www.cs.huji.ac.il/~gkindler/}.

\bibitem{KMS12}
J\'{a}nos K\"{o}rner, Silvia Messuti, and G\'{a}bor Simonyi, \emph{Families of
  graph-different {H}amilton paths}, SIAM J. Discrete Math. \textbf{26} (2012),
  321--329.

\bibitem{McK}
Brendan McKay, \emph{Combinatorial data},
  \url{https://users.cecs.anu.edu.au/~bdm/data/}.

\end{thebibliography}
